\documentclass[12pt]{article}
\usepackage{amssymb}
\usepackage{amsthm}
\usepackage{amsmath}
\usepackage{graphicx}
\usepackage{enumerate}
\usepackage{pict2e}
\usepackage{framed}

\DeclareMathOperator{\Cl}{Cl}

\newtheorem{theorem}{Theorem}[section]

\newtheorem{lemma}[theorem]{Lemma}
\newtheorem{proposition}[theorem]{Proposition}

\newtheorem{remark}[theorem]{Remark}
\newtheorem{example}[theorem]{Example}
\newtheorem{corollary}[theorem]{Corollary}
\title{The operator of relative complementation}
\author{Ivan~Chajda and Helmut~L\"anger}
\date{}

\begin{document}
	
\footnotetext{Support of the research of the first author by the Czech Science Foundation (GA\v CR), project 24-14386L, entitled ``Representation of algebraic semantics for substructural logics'', and by IGA, project P\v rF~2024~011, is gratefully acknowledged.}

\maketitle
	
\begin{abstract}
By the operator of relative complementation is meant a mapping assigning to every element $x$ of an interval $[a,b]$ of a lattice $\mathbf L$ the set $x^{ab}$ of all relative complements of $x$ in $[a,b]$. Of course, if $\mathbf L$ is relatively complemented then $x^{ab}$ is non-empty for each interval $[a,b]$ and every element $x$ belonging to it. We study the question under what condition a complement of $x$ in $\mathbf L$ induces a relative complement of $x$ in $[a,b]$ It is well-known that this is the case provided $\mathbf L$ is modular and complemented. However, we present a more general result. Further, we investigate properties of the operator of relative complementation, in particular in the case when the interval $[a,b]$ is a modular sublattice of $\mathbf L$ or if it is finite. Moreover, we characterize when the operator of relative complementation is involutive and we show a class of lattices where this identity holds. Finally, we establish sufficient conditions under which two different complements of a given element $x$ of $[a,b]$ induce the same relative complement of $x$ in this interval.
\end{abstract}
	
{\bf AMS Subject Classification:} 06C20, 06C15, 06C05, 06A15
	
{\bf Keywords:} Complemented lattice, relative complement, modular lattice, operator of relative complementation

\section{Introduction}

Let $\mathbf L=(L,\vee,\wedge,0,1)$ be a bounded lattice and $a\in L$. The element $b$ of $L$ is called a {\em complement} of $a$ if $a\vee b=1$ and $a\wedge b=0$. The lattice $\mathbf L$ is called {\em complemented} if every of its elements has at least one complement. It is well known (cf.\ \cite B) that if $\mathbf L$ is distributive then every element of $L$ has at most one complement. For more information on complemented lattices the reader is referred to \cite D. Let $\mathbf L=(L,\vee,\wedge)$ be a lattice, $a,b\in L$ with $a\le b$ and $c\in[a,b]$. An element $d$ of $[a,b]$ is called a {\em relative complement} of $c$ in $[a,b]$ if $c\vee d=b$ and $c\wedge d=a$. The lattice $\mathbf L$ is called {\em relatively complemented} if for every $a,b\in L$ with $a\le b$ every element of $[a,b]$ has at least one relative complement in $[a,b]$. It is well known that if $\mathbf L$ is a complemented modular lattice, $a,b\in L$ with $a\le b$, $c\in[a,b]$ and $d$ is a complement of $c$ in $L$ then $(d\vee a)\wedge b=(d\wedge b)\vee a$ is a relative complement of $c$ in $[a,b]$. The question is if all relative complements of $c$ in $[a,b]$ arise in this way and under what conditions this method works also in the non-modular case.

In our previous paper \cite{CL} we introduced the operator $^+$ on a bounded lattice $\mathbf L=(L,\vee,\wedge,0,1)$ as follows:
\begin{align*}
x^+ & :=\{y\in L\mid x\vee y=1\text{ and }x\wedge y=0\}, \\
A^+ & :=\{y\in L\mid x\vee y=1\text{ and }x\wedge y=0\text{ for all }x\in A\}
\end{align*}
for all $x\in L$ and all subsets $A$ of $L$. Of course, $\mathbf L$ is complemented if and only if $x^+\ne\emptyset$ for all $x\in L$.

Now we introduce an analogous operator for relative complements. Let $\mathbf L=(L,\vee,\wedge)$ be a lattice and $a,b\in L$ with $a\le b$. We define
\begin{align*}
x^{ab} & :=\{y\in[a,b]\mid x\vee y=b\text{ and }x\wedge y=a\}, \\
A^{ab} & :=\{y\in[a,b]\mid x\vee y=b\text{ and }x\wedge y=a\text{ for all }x\in A\}
\end{align*}
for all $x\in[a,b]$ and all subsets $A$ of $[a,b]$. Of course, $\mathbf L$ is relatively complemented if and only if $x^{ab}\ne\emptyset$ for all $a,b\in L$ with $a\le b$ and all $x\in[a,b]$.

Throughout the paper we often identify singletons with their unique element.

\section{Complemented lattices}

Let $\mathbf L=(L,\vee,\wedge,0,1)$ be a complemented lattice and $a,b\in L$ with $a\le b$. Put
\begin{align*}
 \overline x_{ab} & :=(x^+\vee a)\wedge b, \\
\widehat x_{ab} & :=(x^+\wedge b)\vee a
\end{align*}
for all $x\in L$ where for any subsets $A,B$ of $L$
\begin{align*}
  A\vee B & :=\{x\vee y\mid x\in A,y\in B\}, \\
A\wedge B & :=\{x\wedge y\mid x\in A,y\in B\}.
\end{align*}

The following result is well known.

\begin{lemma}
If $\mathbf L=(L,\vee,\wedge,0,1)$ is a modular complemented lattice and $a,b\in L$ with $a\le b$ then $\overline x_{ab}=\widehat x_{ab}\subseteq x^{ab}$ for all $x\in[a,b]$ and hence $\mathbf L$ is relatively complemented.
\end{lemma}

An element $v\in[x,y]$ is called {\em induced} by $u\in z^+$ if $v=(u\vee x)\wedge y$, respectively $v=(u\wedge y)\vee x$. We are going to characterize whether an element $v\in[x,y]$ induced by some $u\in z^+$ belongs to $z^{xy}$ . It turns out that this construction is more general, i.e.\ instead of $u$ there can be taken also an element not belonging to $z^+$ provided it satisfies some easy condition.

\begin{proposition}\label{prop3}
Let $\mathbf L=(L,\vee,\wedge)$ be a lattice, $x,y\in L$ with $x\le y$, $z\in[x,y]$ and $u\in L$, assume
\begin{enumerate}
\item[{\rm(1)}] $(u\vee x)\wedge y=(u\wedge y)\vee x$
\end{enumerate}
and put $v:=(u\vee x)\wedge y$. Then $v\in z^{xy}$ if and only if
\begin{enumerate}
\item[{\rm(2)}] $(u\vee x)\wedge z=x$ and $(u\wedge y)\vee z=y$.
\end{enumerate}
\end{proposition}

\begin{proof}
It is evident that $v\in[x,y]$. If $v\in z^{xy}$ then
\begin{align*}
(u\vee x)\wedge z & =z\wedge\big((u\vee x)\wedge y\big)=z\wedge v=x, \\
(u\wedge y)\vee z & =z\vee\big((u\wedge y)\vee x\big)=z\vee v=y,
\end{align*}
i.e.\ (2) holds. If, conversely, (2) holds then
\begin{align*}
  z\vee v & =z\vee\big((u\wedge y)\vee x\big)=(u\wedge y)\vee z=y, \\
z\wedge v & =z\wedge\big((u\vee x)\wedge y\big)=(u\vee x)\wedge z=x,
\end{align*}
i.e.\ $v\in z^{xy}$.
\end{proof}

\begin{remark}
It is worth noticing that if $\mathbf L$ is modular then for $a,b\in L$ with $a\le b$ and $d\in L$ condition {\rm(1)} holds automatically. Hence, for $c\in[a,b]$ we have that $(d\vee a)\wedge b\in c^{ab}$ if and only if condition {\rm(2)} is satisfied.
\end{remark}

We can check the conditions from Proposition~\ref{prop3} in the following example.

\begin{example}\label{ex1}
Consider the non-modular complemented lattice $(L,\vee,\wedge)$ depicted in Fig.~1:

\vspace*{-3mm}

\begin{center}
\setlength{\unitlength}{7mm}
\begin{picture}(8,12)
\put(2,1){\circle*{.3}}
\put(3,3){\circle*{.3}}
\put(1,4){\circle*{.3}}
\put(3,5){\circle*{.3}}
\put(2,7){\circle*{.3}}
\put(6,5){\circle*{.3}}
\put(7,7){\circle*{.3}}
\put(5,8){\circle*{.3}}
\put(7,9){\circle*{.3}}
\put(6,11){\circle*{.3}}
\put(2,1){\line(-1,3)1}
\put(2,1){\line(1,2)1}
\put(3,3){\line(0,1)2}
\put(2,7){\line(-1,-3)1}
\put(2,7){\line(1,-2)1}
\put(6,5){\line(-1,3)1}
\put(6,5){\line(1,2)1}
\put(7,7){\line(0,1)2}
\put(6,11){\line(-1,-3)1}
\put(6,11){\line(1,-2)1}
\put(2,1){\line(1,1)4}
\put(3,3){\line(1,1)4}
\put(1,4){\line(1,1)4}
\put(3,5){\line(1,1)4}
\put(2,7){\line(1,1)4}
\put(1.85,.3){$0$}
\put(.4,3.85){$b$}
\put(2.4,2.85){$a$}
\put(2.4,4.85){$c$}
\put(6.3,4.85){$e$}
\put(1.4,6.85){$d$}
\put(4.4,7.85){$g$}
\put(7.3,6.85){$f$}
\put(7.3,8.85){$h$}
\put(5.85,11.4){$1$}
\put(3.2,-.75){{\rm Fig.~1}}
\put(-.5,-1.75){{\rm Non-modular complemented lattice}}
\end{picture}
\end{center}

\vspace*{10mm}

We can check the following cases.

Consider $g\in[b,1]$. Then $g^+=\{a,c\}$ and for $a\in g^+$ we have $(a\vee b)\wedge1=d=(a\wedge1)\vee b$ and hence {\rm(1)} is satisfied. Moreover, $(a\vee b)\wedge g=d\wedge g=b$ and $(a\wedge1)\vee g=1$ proving {\rm(2)}. Hence, by Proposition~\ref{prop3}, $d\in g^{b1}$. Further, for $c\in g^+$ we have $(c\vee b)\wedge1=d=(c\wedge1)\vee b$ and hence {\rm(1)} is satisfied. Moreover, $(c\vee b)\wedge g=d\wedge g=b$ and $(c\wedge1)\vee g=1$ proving {\rm(2)}. Hence, by Proposition~\ref{prop3}, $d\in g^{b1}$.

Consider $f\in[e,h]$. Clearly, $f^+=b$ and for $b\in f^+$ we have $(b\vee e)\wedge h=g\wedge h=e=0\vee e=d=(b\wedge h)\vee e$ and hence {\rm(1)} is satisfied. Moreover, $(b\vee e)\wedge f=g\wedge f=e$ and $(b\wedge h)\vee f=0\vee f=f\ne h$ showing that {\rm(2)} does not hold. Hence, by Proposition~\ref{prop3}, $e\notin f^{eh}$.

Consider $f\in[e,1]$. Clearly, $f^+=b$ and for $b\in f^+$ we have $(b\vee e)\wedge1=g=(b\wedge1)\vee e$ and hence {\rm(1)} is satisfied. Moreover, $(b\vee e)\wedge f=g\wedge f=e$ and $(b\wedge1)\vee f=1$ proving {\rm(2)}. Hence, by Proposition~\ref{prop3}, $g\in f^{e1}$.

Consider $b\in[0,d]$. Then $b^+=\{f,h\}$ and for $f\in b^+$ we have $(f\vee0)\wedge d=a=(f\wedge d)\vee0$ and hence {\rm(1)} is satisfied. Moreover, $(f\vee0)\wedge b=0$ and $(f\wedge d)\vee b=a\vee b=d$ proving {\rm(2)}. Hence, by Proposition~\ref{prop3}, $a\in b^{0d}$. Further, for $h\in b^+$ we have $(h\vee0)\wedge d=c=(h\wedge d)\vee0$ and hence {\rm(1)} is satisfied. Moreover, $(h\vee0)\wedge b=0$ and $(h\wedge d)\vee b=c\vee b=d$ proving {\rm(2)}. Hence, by Proposition~\ref{prop3}, $c\in b^{0d}$.

Consider $f\in[a,h]$. We have $(b\vee a)\wedge h=d\wedge h=c\ne a=0\vee a=(b\wedge h)\vee a$ and hence {\rm(1)} is not satisfied. Therefore Proposition~\ref{prop3} cannot be applied in this case. Nevertheless, $f^{ah}=\{c\}$.
\end{example}

\section{Properties of the operator $^{ab}$}

Properties of the operator of relative complementation are in certain sense similar to those of the operator of complementation already investigated in by the authors in \cite{CL}. Hence, also the proofs of the following results are written in a similar way, in fact in some cases we only rename the operators. However, for the reader's convenience, the full proofs are presented.

Throughout this section let $(L,\vee,\wedge)$ be a lattice and $a,b\in L$ with $a<b$ and assume $[a,b]$ to be complemented.

The pair $(^{ab},^{ab})$ is the Galois connection between $(2^{[a,b]},\subseteq)$ and $(2^{[a,b]},\subseteq)$ induced by the relation
\[
\{(x,y)\in [a,b]^2\mid x\vee y=b\text{ and }x\wedge y=a\}.
\]
From this we conclude
\begin{align*}
                           A & \subseteq (A^{ab})^{ab}, \\
                A\subseteq B & \Rightarrow B^{ab}\subseteq A^{ab}, \\
\big((A^{ab})^{ab}\big)^{ab} & =A^{ab}, \\
           A\subseteq B^{ab} & \Leftrightarrow B\subseteq A^{ab}
\end{align*}
for all $A,B\subseteq[a,b]$. Since $A\subseteq(A^{ab})^{ab}$ we have that $(A^{ab})^{ab}\ne\emptyset$ whenever $\emptyset\ne A\subseteq[a,b]$. A {\em subset} $A$ of $[a,b]$ is called {\em closed} if $(A^{ab})^{ab}=A$. Let $\Cl([a,b])$ denote the set of all closed subsets of $[a,b]$. Then clearly $\Cl([a,b])=\{A^{ab}\mid A\subseteq[a,b]\}$. If $a<b$ then because of $A^{ab}\cap(A^{ab})^{ab}=\emptyset$ for all $A\subseteq[a,b]$ we have that $\big(\Cl([a,b]),\subseteq,{}^{ab},\emptyset,[a,b]\big)$ forms a complete ortholattice with
\begin{align*}
  \bigvee_{i\in I}A_i & =\left(\left(\bigcup_{i\in I}A_i\right)^{ab}\right)^{ab}, \\
\bigwedge_{i\in I}A_i & =\bigcap_{i\in I}A_i
\end{align*}
for all families $(A_i;i\in I)$ of closed subsets of $[a,b]$.

Next we describe the basic properties of the operator $^{ab}$.

\begin{proposition}\label{prop1}
Let $c\in[a,b]$. Then the following holds:
\begin{enumerate}[{\rm(i)}]
\item $c\in(c^{ab})^{ab}$ and $\big((c^{ab})^{ab}\big)^{ab}=c^{ab}$,
\item $(x^{ab},\le)$ is an antichain for every $x\in[a,b]$ if and only if $[a,b]$ does not contain a sublattice isomorphic to $\mathbf N_5$ containing $a$ and $b$,
\item $(c^{ab},\le)$ is convex,
\item if the mapping $x\mapsto(x^{ab})^{ab}$ from $[a,b]$ to $2^{[a,b]}$ is not injective then the identity $(x^{ab})^{ab}\approx x$ is not satisfied in the interval $[a,b]$.
\end{enumerate}
\end{proposition}

\begin{proof}
\
\begin{enumerate}[(i)]
\item follows directly from above.
\item First assume there exists some $d\in[a,b]$ such that $(d^{ab},\le)$ is not an antichain. Then $d\notin\{a,b\}$. Now there exist $e,f\in d^{ab}$ with $e<f$. Since $d\notin\{a,b\}$ and $d\in e^{ab}\cap f^{ab}$ we have $e,f\notin\{a,b\}$. Because of $|[a,b]|>1$ we have $d\notin\{e,f\}$. Hence the elements $a$, $d$, $e$, $f$ and $b$ are pairwise distinct and form an $\mathbf N_5$ containing $a$ and $b$. If, conversely, $[a,b]$ contains a sublattice $(L_1,\vee,\wedge)$ isomorphic to $\mathbf N_5$ and containing $a$ and $b$, say $L_1=\{a,g,h,i,1\}$ with $g<h$ then $g,h\in i^{ab}$ and hence $(i^{ab},\le)$ is not an antichain.
\item If $d,e\in c^{ab}$, $f\in[a,b]$ and $d\le f\le e$ then $b=c\vee d\le c\vee f$ and $c\wedge f\le c\wedge e=0$ showing $f\in c^{ab}$.
\item If the mapping $x\mapsto(x^{ab})^{ab}$ is not injective then there exist $c,d\in[a,b]$ with $c\ne d$ and $(c^{ab})^{ab}=(d^{ab})^{ab}$ which implies $d\in(d^{ab})^{ab}=(c^{ab})^{ab}$ and $d\ne c$ and hence $(c^{ab})^{ab}\ne c$ showing that $[a,b]$ does not satisfy the identity $(x^{ab})^{ab}\approx x$.
\end{enumerate}
\end{proof}

However, also $A^{ab}$ is an antichain for a non-empty subset $A$ of $[a,b]$ provided $[a,b]$ is a modular sublattice of $\mathbf L$, see the following result.

\begin{corollary}\label{cor1}
Assume $[a,b]$ to be a modular sublattice of $\mathbf L$ and let $c\in[a,b]$ and $A$ be a non-empty subset of $[a,b]$. According to Proposition~\ref{prop1} {\rm(iii)}, $(c^{ab},\le)$ is an antichain. Let $d\in A$. Then $A^{ab}\subseteq d^{ab}$ and hence $(A^{ab},\le)$ is an antichain, too. Since $c^{ab}$ is a non-empty subset of $[a,b]$ we finally conclude that $\big((c^{ab})^{ab},\le\big)$ is an antichain, too.
\end{corollary}

In case of finite $[a,b]$ we can even prove the following.

\begin{proposition}
Assume $[a,b]$ to be finite, the mapping $x\mapsto(x^{ab})^{ab}$ from $[a,b]$ to $2^{[a,b]}$ to be injective and $c\in[a,b]$. Then there exists some $d\in(c^{ab})^{ab}$ with $(d^{ab})^{ab}=d$.
\end{proposition}

\begin{proof}
If $(c^{ab})^{ab}=c$ then put $d:=c$. Now assume $(c^{ab})^{ab}\ne c$. Let $c_1\in(c^{ab})^{ab}\setminus\{c\}$. Then $(c_1^{ab})^{ab}\subseteq(c^{ab})^{ab}$. Since $c_1\ne c$ and $x\mapsto(x^{ab})^{ab}$ is injective we conclude $(c_1^{ab})^{ab}\subsetneqq(c^{ab})^{ab}$. Now either $(c_1^{ab})^{ab}=c_1$ or there exists some $c_2\in(c_1^{ab})^{ab}\setminus\{c_1\}$. In the latter case $(c_2^{ab})^{ab}\subseteq(c_1^{ab})^{ab}$. Since $c_2\ne c_1$ and $x\mapsto(x^{ab})^{ab}$ is injective we conclude $(c_2^{ab})^{ab}\subsetneqq(c_1^{ab})^{ab}$. Now either $(c_2^{ab})^{ab}=c_2$ or there exists some $c_3\in(c_2^{ab})^{ab}\setminus\{c_2\}$. Since $[a,b]$ is finite and $(c_1^{ab})^{ab}\supsetneqq(c_2^{ab})^{ab}\supsetneqq\cdots$ there exists some $n\ge1$ with $|(c_n^{ab})^{ab}|=1$, i.e.\ $(c_n^{ab})^{ab}=c_n$ and we have $c_n\in(c_n^{ab})^{ab}\subseteq(c_{n-1}^{ab})^{ab}\subseteq\cdots\subseteq(c_1^{ab})^{ab}\subseteq(c^{ab})^{ab}$.
\end{proof}

Let $(P,\le)$ be a poset and $A,B\subseteq P$. Define $A\le_1B$ if for every $a\in A$ there exists some $b\in B$ with $a\le b$. Clearly, $\le_1$ is a reflexive and transitive relation on $2^P$ and, moreover, $A\subseteq B$ implies $A\le_1B$. We define $A=_1B$ if $A\le_1B$ as well as $B\le_1A$.

The relationship between the operator $^{ab}$ and the partial order relation on $[a,b]$ is illuminated in the following result.

\begin{theorem}
Consider the following statements:
\begin{enumerate}[{\rm(i)}]
\item $x^{ab}\vee y^{ab}\le_1(x\wedge y)^{ab}$ for all $x,y\in[a,b]$,
\item for all $x,y\in[a,b]$, $x\le y$ implies $y^{ab}\le_1x^{ab}$,
\item $(x\vee y)^{ab}\le_1x^{ab}\wedge y^{ab}$ for all $x,y\in[a,b]$,
\item $x^{ab}\vee y^{ab}\subseteq(x\wedge y)^{ab}$ for all $x,y\in[a,b]$,
\item $x^{ab}\wedge y^{ab}\subseteq(x\vee y)^{ab}$ for all $x,y\in[a,b]$,
\item $(x\vee y)^{ab}=_1x^{ab}\wedge y^{ab}$ for all $x,y\in[a,b]$.
\end{enumerate}
Then {\rm(i)} implies {\rm(ii)}, {\rm(ii)} and {\rm(iii)} are equivalent, {\rm(iv)} implies {\rm(i)} -- {\rm(iii)}, and {\rm(iv)} and {\rm(v)} together imply {\rm(vi)}.
\end{theorem}

\begin{proof}
Let $c,d\in[a,b]$. \\
(i) implies (ii): \\
$c\le d$ implies $d^{ab}\le_1c^{ab}\vee d^{ab}\le(c\wedge d)^{ab}=c^{ab}$. \\
(ii) implies (iii): \\
Because of $c,d\le c\vee d$ we have $(c\vee d)^{ab}\le_1c^{ab},d^{ab}$ which implies $(c\vee d)^{ab}\le_1c^{ab}\wedge d^{ab}$. \\
(iii) implies (ii): \\
$c\le d$ implies $d^{ab}=(c\vee d)^{ab}\le_1c^{ab}\wedge d^{ab}\le_1c^{ab}$. \\
(iv) implies (i) -- (iii): \\
We have (i) and hence also (ii) and (iii) are satisfied. \\
(iv) and (v) together imply (vi): \\
We have seen that (iv) implies (iii). Now $x^{ab}\wedge y^{ab}\le_1(x\vee y)^{ab}$ for all $x,y\in[a,b]$ which together with (iii) yields (vi).
\end{proof}

\section{Modular relatively complemented lattices}

Our next task is to characterize the property that $[a,b]$ satisfies the identity $(x^{ab})^{ab}\approx x$.

As mentioned in Example~\ref{ex1} and in the previous section, this task has a deeper sense only if the interval $[a,b]$ in question is assumed to be a modular sublattice of $\mathbf L$.

Again, we can modify our results from \cite{CL}.

\begin{theorem}\label{th2}
Let $\mathbf L=(L,\vee,\wedge)$ be a lattice, $a,b\in L$ with $a\le b$ and assume $[a,b]$ to be a modular sublattice of $\mathbf L$. Then the following are equivalent:
\begin{enumerate}[{\rm(i)}]
\item $[a,b]$ satisfies the identity $(x^{ab})^{ab}\approx x$,
\item for every $x\in[a,b]$ and each $y\in(x^{ab})^{ab}$ there exists some $z\in y^{ab}$ satisfying either $(x\vee y)\wedge z=a$ or $(x\wedge y)\vee z=b$.
\end{enumerate}
\end{theorem}

\begin{proof}
$\text{}$ \\
(i) $\Rightarrow$ (ii): \\
If $c\in[a,b]$, $d\in(c^{ab})^{ab}$ and $e\in d^{ab}$ then $d=c$ and $(c\vee c)\wedge e=d\wedge e=0$. \\
(ii) $\Rightarrow$ (i): \\
Suppose $[a,b]$ not to satisfy the identity $(x^{ab})^{ab}\approx x$. Then there exists some $c\in[a,b]$ with $(c^{ab})^{ab}\ne c$. Let $\in(c^{ab})^{ab}\setminus\{c\}$. According to (ii) there exists some $e\in d^{ab}$ satisfying either $(c\vee d)\wedge e=a$ or $(c\wedge d)\vee e=b$. Since $c$ and $d$ are different elements of $(c^{ab})^{ab}$ and $\big((c^{ab})^{ab},\le\big)$ is an antichain according to Corollary~\ref{cor1}, we conclude $c\parallel d$. Now $(c\vee d)\wedge e=a$ would imply
\[
c\le c\vee d=b\wedge(c\vee d)=(d\vee e)\wedge(c\vee d)=d\vee\big(e\wedge(c\vee d)\big)=d\vee a=d
\]
contradicting $c\parallel d$. On the other hand, $(c\wedge d)\vee e=b$ would imply
\[
d=b\wedge d=\big((c\wedge d)\vee e\big)\wedge d=(c\wedge d)\vee(e\wedge d)=(c\wedge d)\vee a=c\wedge d\le c
\]
again contradicting $c\parallel d$. This shows that $[a,b]$ satisfies the identity $(x^{ab})^{ab}\approx x$.
\end{proof}

This result can be demonstrated by the following example.

\begin{example}
Consider the modular relatively complemented lattice $\mathbf L=(L,\vee,\wedge)$ visualized in Fig.~2:

\vspace*{-3mm}

\begin{center}
\setlength{\unitlength}{7mm}
\begin{picture}(14,8)
\put(7,1){\circle*{.3}}
\put(1,3){\circle*{.3}}
\put(3,3){\circle*{.3}}
\put(5,3){\circle*{.3}}
\put(7,3){\circle*{.3}}
\put(9,3){\circle*{.3}}
\put(11,3){\circle*{.3}}
\put(13,3){\circle*{.3}}
\put(1,5){\circle*{.3}}
\put(3,5){\circle*{.3}}
\put(5,5){\circle*{.3}}
\put(7,5){\circle*{.3}}
\put(9,5){\circle*{.3}}
\put(11,5){\circle*{.3}}
\put(13,5){\circle*{.3}}
\put(7,7){\circle*{.3}}
\put(7,1){\line(-3,1)6}
\put(7,1){\line(-2,1)4}
\put(7,1){\line(-1,1)2}
\put(7,1){\line(0,1)6}
\put(7,1){\line(1,1)2}
\put(7,1){\line(2,1)4}
\put(7,1){\line(3,1)6}
\put(7,7){\line(-3,-1)6}
\put(7,7){\line(-2,-1)4}
\put(7,7){\line(-1,-1)2}
\put(7,7){\line(1,-1)2}
\put(7,7){\line(2,-1)4}
\put(7,7){\line(3,-1)6}
\put(1,3){\line(0,1)2}
\put(1,3){\line(1,1)2}
\put(1,3){\line(2,1)4}
\put(3,3){\line(-1,1)2}
\put(3,3){\line(2,1)4}
\put(3,3){\line(3,1)6}
\put(5,3){\line(-2,1)4}
\put(5,3){\line(3,1)6}
\put(5,3){\line(4,1)8}
\put(7,3){\line(-2,1)4}
\put(7,3){\line(2,1)4}
\put(9,3){\line(-3,1)6}
\put(9,3){\line(0,1)2}
\put(9,3){\line(2,1)4}
\put(11,3){\line(-3,1)6}
\put(11,3){\line(-2,1)4}
\put(11,3){\line(1,1)2}
\put(13,3){\line(-4,1)8}
\put(13,3){\line(-2,1)4}
\put(13,3){\line(-1,1)2}
\put(6.85,.3){$0$}
\put(.4,2.85){$a$}
\put(2.4,2.85){$b$}
\put(4.4,2.85){$c$}
\put(7.3,2.85){$d$}
\put(9.3,2.85){$e$}
\put(11.3,2.85){$f$}
\put(13.3,2.85){$g$}
\put(.4,4.85){$h$}
\put(2.4,4.85){$i$}
\put(4.4,4.85){$j$}
\put(7.3,4.85){$k$}
\put(9.3,4.85){$l$}
\put(11.3,4.85){$m$}
\put(13.3,4.85){$n$}
\put(6.85,7.4){$1$}
\put(6.2,-.75){{\rm Fig.~2}}
\put(1.6,-1.75){{\rm Modular relatively complemented lattice}}
\end{picture}
\end{center}

\vspace*{10mm}

The lattice $\mathbf L$ is isomorphic to the lattice of linear subspaces of the three-dimensional vector space over the two-element field and hence $\mathbf L$ is modular and complemented. The join of two different atoms is a coatom and the meet of two different coatoms is an atom. Every atom is covered by three coatoms and every coatom covers three atoms. Now
\begin{align*}
\overline a_{0h} & =\widehat a_{0h}=(\{k,l,m,n\}\vee0)\wedge h=\{b,c\}=a^{0h}, \\
   (a^{0h})^{0h} & =\{b,c\}^{0h}=a.
\end{align*}

Similarly we can proceed in the remaining cases. Hence, {\rm(ii)} of Theorem~\ref{th2} is satisfied. Thus we have
\begin{align*}
\overline z_{xy} & =\widehat z_{xy}=z^{xy}, \\
   (z^{xy})^{xy} & =z
\end{align*}
for all $x,y\in L$ with $x\le y$ and all $z\in[x,y]$.
\end{example}

In what follows we present a class of lattices where condition (ii) of Theorem~\ref{th2} holds and hence identity (i) is satisfied.

For $n\ge3$ let $\mathbf M_n$ denote the modular lattice depicted in Fig.~3:

\vspace*{-3mm}

\begin{center}
\setlength{\unitlength}{7mm}
\begin{picture}(8,6)
\put(4,1){\circle*{.3}}
\put(1,3){\circle*{.3}}
\put(3,3){\circle*{.3}}
\put(7,3){\circle*{.3}}
\put(4,5){\circle*{.3}}
\put(4,1){\line(-3,2)3}
\put(4,1){\line(-1,2)1}
\put(4,1){\line(3,2)3}
\put(4,5){\line(-3,-2)3}
\put(4,5){\line(-1,-2)1}
\put(4,5){\line(3,-2)3}
\put(3.85,.3){$0$}
\put(.25,2.85){$a_1$}
\put(2.25,2.85){$a_2$}
\put(4.5,2.85){$\cdots$}
\put(7.3,2.85){$a_n$}
\put(3.85,5.4){$1$}
\put(3.2,-.75){{\rm Fig.~3}}
\put(-2.3,-1.75){The modular relatively complemented lattice $\mathbf M_n$}
\end{picture}
\end{center}

\vspace*{10mm}

We can easily prove

\begin{theorem}\label{th1}
Let $\mathbf L$ be the direct product of a Boolean algebra and an arbitrary number of lattices $\mathbf M_n$ for possibly various positive integers $n$, $a,b\in L$ with $a\le b$ and $c\in[a,b]$. Then the following holds:
\begin{enumerate}[{\rm(i)}]
\item $(c^{ab})^{ab}=c$,
\item $\overline c_{ab}=\widehat c_{ab}=c^{ab}$.
\end{enumerate}
\end{theorem}

\begin{proof}
Since Boolean algebras and lattices of the form $\mathbf M_n$ have properties (i) and (ii) and these properties are preserved by forming direct products, also $\mathbf L$ has these properties.
\end{proof}

The following example illuminates Theorem~\ref{th1}.

\begin{example}
Let $\mathbf2$ denote the two-element Boolean algebra. Consider the modular relatively complemented lattice $\mathbf L_1=(L_1,\vee,\wedge)=\mathbf 2\times\mathbf M_3$ visualized in Fig.~4:	

\vspace*{-3mm}

\begin{center}
\setlength{\unitlength}{7mm}
\begin{picture}(12,8)
\put(3,1){\circle*{.3}}
\put(1,3){\circle*{.3}}
\put(3,3){\circle*{.3}}
\put(5,3){\circle*{.3}}
\put(3,5){\circle*{.3}}
\put(9,3){\circle*{.3}}
\put(7,5){\circle*{.3}}
\put(9,5){\circle*{.3}}
\put(11,5){\circle*{.3}}
\put(9,7){\circle*{.3}}
\put(3,1){\line(-1,1)2}
\put(3,1){\line(0,1)4}
\put(3,1){\line(1,1)2}
\put(3,5){\line(-1,-1)2}
\put(3,5){\line(1,-1)2}
\put(9,3){\line(-1,1)2}
\put(9,3){\line(0,1)4}
\put(9,3){\line(1,1)2}
\put(9,7){\line(-1,-1)2}
\put(9,7){\line(1,-1)2}
\put(3,1){\line(3,1)6}
\put(1,3){\line(3,1)6}
\put(3,3){\line(3,1)6}
\put(5,3){\line(3,1)6}
\put(3,5){\line(3,1)6}
\put(2.85,.3){$0$}
\put(.4,2.85){$a$}
\put(2.4,2.85){$b$}
\put(4.4,2.85){$c$}
\put(2.85,5.4){$d$}
\put(8.85,2.3){$e$}
\put(7.3,4.85){$f$}
\put(9.3,4.85){$g$}
\put(11.3,4.85){$h$}
\put(8.85,7.4){$1$}
\put(5.2,-.75){{\rm Fig.~4}}
\put(.7,-1.75){{\rm Modular relatively complemented lattice}}
\end{picture}
\end{center}

\vspace*{10mm}

It is easy to check that $(z^{xy})^{xy}=z$ for all $x,y\in L_1$ with $x\le y$ and for all $z\in[x,y]$ as pointed out in Theorem~\ref{th1}.
\end{example}

We now present an example of a relatively complemented lattice $(L,\vee,\wedge)$ not satisfying $\overline c_{ab}=c^{ab}$ for all $a,b\in L$ with $a\le b$ and all $c\in[a,b]$.

\begin{example}
Consider the non-modular relatively complemented lattice $\mathbf L_2=(L,\vee,\wedge)$ depicted in Fig.~5:

\vspace*{-3mm}

\begin{center}
\setlength{\unitlength}{7mm}
\begin{picture}(14,8)
\put(3,1){\circle*{.3}}
\put(1,3){\circle*{.3}}
\put(3,3){\circle*{.3}}
\put(5,3){\circle*{.3}}
\put(3,5){\circle*{.3}}
\put(9,3){\circle*{.3}}
\put(7,5){\circle*{.3}}
\put(9,5){\circle*{.3}}
\put(11,5){\circle*{.3}}
\put(13,5){\circle*{.3}}
\put(9,7){\circle*{.3}}
\put(3,1){\line(-1,1)2}
\put(3,1){\line(0,1)4}
\put(3,1){\line(1,1)2}
\put(3,5){\line(-1,-1)2}
\put(3,5){\line(1,-1)2}
\put(9,3){\line(-1,1)2}
\put(9,3){\line(0,1)4}
\put(9,3){\line(1,1)2}
\put(9,3){\line(2,1)4}
\put(9,7){\line(-1,-1)2}
\put(9,7){\line(1,-1)2}
\put(9,7){\line(2,-1)4}
\put(3,1){\line(3,1)6}
\put(1,3){\line(3,1)6}
\put(3,3){\line(3,1)6}
\put(5,3){\line(3,1)6}
\put(3,5){\line(3,1)6}
\put(2.85,.3){$0$}
\put(.4,2.85){$a$}
\put(2.4,2.85){$b$}
\put(4.4,2.85){$c$}
\put(2.85,5.4){$d$}
\put(8.85,2.3){$e$}
\put(7.3,4.85){$f$}
\put(9.3,4.85){$g$}
\put(11.3,4.85){$h$}
\put(13.3,4.85){$i$}
\put(8.85,7.4){$1$}
\put(6.2,-.75){{\rm Fig.~5}}
\put(1,-1.75){{\rm Non-modular relatively complemented lattice}}
\end{picture}
\end{center}

\vspace*{10mm}

Evidently, the lattice $\mathbf L_1$ from Fig.~4 is a sublattice of $\mathbf L_2$. Moreover, $\mathbf L_2$ is not modular since $\{0,d,e,i,1\}$ forms a forbidden sublattice $\mathbf N_5$. On the other hand, the interval $[e,1]$ is a modular sublattice of $\mathbf L_2$. We have
\begin{align*}
         	 h^+ & =\{a,b\}, \\
\overline h_{e1} & =(\{a,b\}\vee e)\wedge1=\{f,g\}, \\
 \widehat h_{e1} & =(\{a,b\}\wedge1)\vee e=\{f,g\}, \\
          h^{e1} & =\{f,g,i\}.
\end{align*}
Observe that the relative complement $i$ of $h$ in $[e,1]$ is neither contained in $\overline h_{e1}$ nor in $\widehat h_{e1}$ since $i$ is both join- and meet-irreducible. Further,
\begin{align*}
             i^+ & =\{a,b,c,d\}, \\
\overline i_{e1} & =(\{a,b,c,d\}\vee e)\wedge1=\{f,g,h,1\}, \\
 \widehat i_{e1} & =(\{a,b,c,d\}\wedge1)\vee e=\{f,g,h,1\}, \\
          i^{e1} & =\{f,g,h\}.
\end{align*}
Observe that the element $1$ of $\overline i_{e1}$ and $\widehat i_{e1}$ is not a relative complement of $i$ in $[e,1]$.
\end{example}

\begin{remark}
Let $\mathbf L=(L,\vee,\wedge)$ be a lattice and $a,b,c\in L$ with $a<b<c$. We ask whether it is possible that two different complements of $c$ induce the same element of $\overline c_{ab}$, respectively $\widehat c_{ab}$. We present two sufficient conditions as follows.

If $\mathbf L$ contains a sublattice of the form visualized in Fig.~6

\vspace*{-3mm}

\begin{center}
\setlength{\unitlength}{7mm}
\begin{picture}(8,8)
\put(3,1){\circle*{.3}}
\put(1,3){\circle*{.3}}
\put(3,3){\circle*{.3}}
\put(5,3){\circle*{.3}}
\put(3,5){\circle*{.3}}
\put(4,4){\circle*{.3}}
\put(5,7){\circle*{.3}}
\put(6,6){\circle*{.3}}
\put(7,5){\circle*{.3}}
\put(3,1){\line(-1,1)2}
\put(3,1){\line(0,1)4}
\put(3,1){\line(1,1)4}
\put(1,3){\line(1,1)4}
\put(4,4){\line(1,1)2}
\put(5,3){\line(-1,1)2}
\put(7,5){\line(-1,1)2}
\put(2.85,.3){$0$}
\put(.4,2.85){$d$}
\put(3.3,2.85){$e$}
\put(5.3,2.85){$a$}
\put(4.85,7.4){$1$}
\put(7.3,4.85){$c$}
\put(6.3,5.85){$b$}
\put(3.2,-.75){{\rm Fig.~6}}
\put(1,-1.75){{\rm Possible sublattice of $\mathbf L$}}
\end{picture}
\end{center}

\vspace*{10mm}

then $d$ and $e$ are different complements of $c$ the induced elements $(d\vee a)\wedge b$ and $(e\vee a)\wedge b$ of $\overline c_{ab}$ of which coincide and this element is a relative complement of $c$ in $[a,b]$.

Similarly, if $\mathbf L$ contains a sublattice of the form depicted in Fig.~7

\vspace*{-3mm}

\begin{center}
\setlength{\unitlength}{7mm}
\begin{picture}(8,8)
\put(3,1){\circle*{.3}}
\put(1,3){\circle*{.3}}
\put(2,2){\circle*{.3}}
\put(5,3){\circle*{.3}}
\put(3,5){\circle*{.3}}
\put(4,4){\circle*{.3}}
\put(5,7){\circle*{.3}}
\put(5,5){\circle*{.3}}
\put(7,5){\circle*{.3}}
\put(3,1){\line(-1,1)2}
\put(2,2){\line(1,1)2}
\put(3,1){\line(1,1)4}
\put(1,3){\line(1,1)4}
\put(5,3){\line(0,1)4}
\put(5,3){\line(-1,1)2}
\put(7,5){\line(-1,1)2}
\put(2.85,.3){$0$}
\put(.4,2.85){$c$}
\put(1.4,1.85){$a$}
\put(2.4,4.85){$b$}
\put(7.3,4.85){$e$}
\put(5.3,4.85){$d$}
\put(4.85,7.4){$1$}
\put(3.2,-.75){{\rm Fig.~7}}
\put(1,-1.75){{\rm Possible sublattice of $\mathbf L$}}
\end{picture}
\end{center}

\vspace*{10mm}

then $d$ and $e$ are different complements of $c$ the induced elements $(d\wedge b)\vee a$ and $(e\wedge b)\vee a$ of $\widehat c_{ab}$ of which coincide and this element is a relative complement of $c$ in $[a,b]$.

This is in accordance with the following fact. Consider the lattice $\mathbf L=(L,\vee,\wedge)$ from Example~\ref{ex1} and let $x,y,z\in L$ with $x<z<y$. Then any two distinct complements of $z$ induce two distinct elements of $\overline z_{xy}$, respectively $\widehat z_{xy}$, and $\mathbf L$ does not contain a sublattice of the described form.
\end{remark}

%\section{Declarations}

%{\bf Compliance with Ethical Standards} This article does not contain any studies with human participants or animals performed by any of the authors.

%{\bf Funding} Support of the research of the first author by the Czech Science Foundation (GA\v CR), project 24-14386L, entitled ``Representation of algebraic semantics for substructural logics'', and by IGA, project P\v rF~2024~011, is gratefully acknowledged.

%{\bf Data availability statement} No datasets were generated or analyzed during the current study.

%{\bf Competing interests} There are no competing interests of a financial or personal nature between the authors.

%{\bf Conflict of interest} Both authors declare that they have no conflict of interest.

%{\bf Authors' contributions} Both authors contributed equally to the manuscript.

Authors' addresses:

Ivan Chajda \\
Palack\'y University Olomouc \\
Faculty of Science \\
Department of Algebra and Geometry \\
17.\ listopadu 12 \\
771 46 Olomouc \\
Czech Republic \\
ivan.chajda@upol.cz

Helmut L\"anger \\
TU Wien \\
Faculty of Mathematics and Geoinformation \\
Institute of Discrete Mathematics and Geometry \\
Wiedner Hauptstra\ss e 8-10 \\
1040 Vienna \\
Austria, and \\
Palack\'y University Olomouc \\
Faculty of Science \\
Department of Algebra and Geometry \\
17.\ listopadu 12 \\
771 46 Olomouc \\
Czech Republic \\
helmut.laenger@tuwien.ac.at
\end{document}